\documentclass{amsart}
\usepackage{amssymb,latexsym,amsmath,amsthm,enumitem,mathrsfs,geometry}
\usepackage{fullpage}
\usepackage{fancyhdr}
\usepackage{color}
\usepackage{stmaryrd}
\usepackage{mathrsfs}
\usepackage{hyperref}
\usepackage{lineno}
\usepackage{diagbox}
\usepackage{array}
\usepackage{tikz}
\usetikzlibrary{arrows}
\usetikzlibrary{positioning}
\usepackage{float}
\usepackage{bbm}
\usepackage{cleveref}

\newtheorem*{thm*}{Theorem}

\newcommand{\beq}{\begin{equation}}
\newcommand{\eeq}{\end{equation}}

\newcommand{\1}{\mathbbm{1}}
\newtheorem{theorem}{Theorem}
\newtheorem{lem}{Lemma}

\newtheorem*{theorem*}{Theorem}
\newtheorem*{WHL}{Weak Hardy-Littlewood Goldbach Conjecture}
\newtheorem*{PPC}{Hardy-Littlewood Prime-Pair Upper Bound Conjecture}
\definecolor{pink}{rgb}{1,.2,.6}
\definecolor{orange}{rgb}{0.7,0.3,0}
\definecolor{blue}{rgb}{.2,.6,.75}
\definecolor{green}{rgb}{.4,.7,.4}
\definecolor{purple}{RGB}{127,0,255}

\numberwithin{equation}{section}

\begin{document}

\title{Note on the Goldbach Conjecture and Landau-Siegel Zeros}

\author[Goldston]{D. A. Goldston}
\address{Department of Mathematics and Statistics, San Jose State University}
\email{daniel.goldston@sjsu.edu}

\author[Suriajaya]{Ade Irma Suriajaya}
\address{Faculty of Mathematics, Kyushu University}
\email{adeirmasuriajaya@math.kyushu-u.ac.jp}
\thanks{$^{*}$ The second author was supported by JSPS KAKENHI Grant Number 18K13400.}

\date{\today}

\begin{abstract}
We generalize the work of Fei, Bhowmik and Halupczok, and Jia relating the Goldbach conjecture to real zeros of Dirichlet $L$-functions.

\end{abstract}

\maketitle

\section{Introduction}

Let
\begin{equation}\label{psi_2} \psi_2(n) = \sum_{m+m'=n} \Lambda(m)\Lambda(m'),\end{equation}
where $\Lambda$ is the von Mangoldt function, defined by $\Lambda(n)=\log p$ if $n=p^m$, $p$ a prime and $m\ge 1$, and $\Lambda(n)= 0$ otherwise. Thus $\psi_2(n)$ counts the \lq \lq Goldbach" representations of $n$ as sums of both primes and prime powers, and these primes are weighted to make them have a \lq\lq density" of 1 on the integers. 

Hardy and Littlewood \cite{HardyLittlewood1922} conjectured that, for $n$ even
\begin{equation}\label{GC} \psi_2(n) \sim \mathfrak{S}(n) n, \quad \text{ as} \quad n\to \infty, \end{equation}
where
\begin{equation}\label{SingProd}
\mathfrak{S}(k) = \begin{cases}
{\displaystyle 2 C_{2} \prod_{\substack{ p \mid k \\ p > 2}} \!\left(\frac{p - 1}{p - 2}\right)} & \mbox{if $k$ is even, $k\neq 0$}, \\
 0 & \mbox{if $k$ is odd}\end{cases}
\end{equation}
and
\begin{equation}\label{eq1.3}
C_{2}
 = \prod_{p > 2}\! \left( 1 - \frac{1}{(p - 1)^{2}}\right)
 = 0.66016\ldots.
\end{equation}
When $n$ is odd then the only possible non-zero terms in the sum in \eqref{psi_2} are when $m$ or $m'$ is a power of 2, and since there are $\ll \log n$ such terms we have
\begin{equation}\label{psi_2odd} \psi_2(n) \ll \log^2 n, \qquad \text{for $n$ odd}. \end{equation}

In this paper we will use the following weaker form of \eqref{GC}.
\begin{WHL} Given a fixed constant $0<\delta <1$, then for sufficiently large even $n$, we have
$ | \psi_2(n)-\mathfrak{S}(n)n| \leq (1-\delta) \mathfrak{S}(n) n $. Equivalently, we have
\begin{equation} \label{WGC} \mathbf{(A)} \quad \delta \mathfrak{S}(n) n 
\le \psi_2(n), \qquad \text{and} \qquad \mathbf{(B)} \quad \psi_2(n) \le (2-\delta)\mathfrak{S}(n) n. \end{equation}
\end{WHL}
\noindent We show this conjecture implies that a sequence of Landau-Siegel zeros can only slowly approach 1.
\begin{theorem}\label{thm1} Assume the Weak Hardy-Littlewood Goldbach Conjecture. Let $q$ be sufficiently large, and suppose that $\chi_1$ is the single real character $({\rm mod}\ q)$, if it exists, for which $L(s,\chi_1)$ has a real zero $\beta_1$ satisfying $1- c/\log q < \beta_1 $ for a certain positive absolute constant $c$. Then we have
$\beta_1 < 1- C(\delta)/\log^2q$,
where $C(\delta)$ is a positive effective constant that depends on $\delta$. In particular, if $\chi(-1) =-1$ then this follows from (A) in \eqref{WGC}, and if $\chi(-1)=1$ then this follows from (B) in \eqref{WGC}.
\end{theorem}
Fei \cite{Fei16} obtained this result when $q$ is taken to be a prime $q\equiv 3 ~({\rm mod}\ 4)$.
Bhowmik and Halupczok \cite{BhowHalup2020} obtained the theorem when $\chi(-1)=-1$; actually they use a slightly different form of (A) in \eqref{WGC} which is weaker when $q$ has many prime factors, and the corresponding result on real zeros is consequently also slightly weaker for these $q$. Jia \cite{Jia20} noticed this in the original preprint of \cite{BhowHalup2020} and then used (A) in \eqref{WGC} to prove precisely the result in the theorem above when $\chi(-1)=-1$. It should also be pointed out that Bhowmik and Halupczok proved their result when the Weak Goldbach Conjecture is allowed to not hold on a certain exceptional set.  In a recent talk Bhowmik mentioned as a goal proving the current theorem when $\chi(-1)=1$ using (B) of \eqref{WGC}.

We follow the method of Fei \cite{Fei16} and Bhowmik and Halupczok \cite{BhowHalup2020}, but make use of power series generating functions instead of exponential generating functions, which makes possible the use of the upper bound (B) in \eqref{WGC}. By rearranging the earlier proofs we are also able to remove the need for Gaussian sums in our proof.

With minor adjustments to our proof, we can prove a form of \Cref{thm1} using a prime-pair conjecture in place of the Goldbach conjecture. For $k\ge 0$ let
\begin{equation} \label{psi_2(x,k)} \psi_2(x,k) := \sum_{n\le x} \Lambda(n)\Lambda(n-k) = \sum_{k<n\le x} \Lambda(n)\Lambda(n-k).\end{equation}
The Hardy-Littlewood conjecture we need here is that, for even $2\le k\le x$,
\begin{equation}  \psi_2(x,k) = \mathfrak{S}(k)(x-k) + o(\mathfrak{S}(k)x). \end{equation}
Note that this is true in the range $x-o(x) \le k \le x$ by a standard sieve result \cite[Cor. 5.8.1]{HR11}.
If $k$ is odd then in the same way we obtained \eqref{psi_2odd} we have
\begin{equation} \label{psi_2(x,k)odd} \psi_2(x,k) \ll \log^2x. \end{equation}
The conjecture we need is an upper bound that is slightly smaller than twice the conjectured main term.
\begin{PPC} Given a fixed constant $0<\delta <1$, then for even $2\le k\le x$ and sufficiently large $x$, we have
\begin{equation} \label{WPPC}  \psi_2(x,k) \le (2-\delta)\mathfrak{S}(k) (x-k)+o(\mathfrak{S}(k)x). \end{equation}
\end{PPC}

\begin{theorem}\label{thm2}  Theorem 1 holds if we replace the Weak Hardy-Littlewood Goldbach Conjecture with the Hardy-Littlewood Prime-Pair Upper Bound Conjecture. 
\end{theorem}

\section{Evaluating \texorpdfstring{$\mathcal{S}(q)$}{S(q)} in two ways}

Here all sums run over the positive integers unless specified otherwise. We use the power series generating function
\begin{equation} \label{Psi} \Psi(z) = \sum_{n} \Lambda(n) z^n ,\qquad z=r e(\alpha), \quad e(\alpha) = e^{2\pi i \alpha}, \end{equation}
for $|z|=r<1$. Squaring, we have 
\begin{equation} \label{Psi^2}  \Psi(z)^2 = \sum_{m,m'}\Lambda(m)\Lambda(m') z^{m+m'} = \sum_{n} \psi_2(n) z^n. \end{equation}
Fei's idea is equivalent in this setting to computing
\begin{equation} \label{Sq} \mathcal{S}(q) := \frac{1}{q} \sum_{a=1}^q \Psi(re(a/q))^2   \end{equation}
in two ways. First, using \eqref{Psi^2} and $ \frac1q\sum_{a=1}^qe(an/q) =\1_{q\mid n} $, we have
\begin{equation} \label{Sq1} \mathcal{S}(q) = \frac{1}{q} \sum_{a=1}^q \sum_n\psi_2(n) r^n e(an/q) = \sum_{\substack{n \\ q \mid n}} \psi_2(n)r^n. \end{equation}
Summing $\psi_2(n)$ over multiples of $q$ is a problem that has already occurred in \cite{Gran1}, \cite{Gran2}, and \cite{Bhowmik-H-M-SGoldbach2019}. 

To obtain the second formula for $S(q)$, as in the circle method we separate the terms in $\Psi(re(a/q))$ according to the arithmetic progressions they belong to modulo $q$. Letting
\begin{equation} \label{Psi(r)} \Psi(r; q, b) := \sum_{\substack{n\\ n\equiv b \,({\rm mod}\, q)} }\Lambda(n) r^n, \end{equation}
we have
\[ \Psi(re(a/q)) = \sum_{b=1}^q \sum_{\substack{n\\ n\equiv b \,({\rm mod}\, q)}}\Lambda(n) r^n e(an/q) = \sum_{b=1}^q e(ab/q)\Psi(r; q, b) ,\]
and therefore
\[ \begin{split}  \mathcal{S}(q) &= \frac{1}{q} \sum_{a=1}^q \sum_{1\le b,b'\le q} e(a(b+b')/q)\Psi(r; q, b)\Psi(r; q, b') \\ &
= \sum_{\substack{1\le b,b'\le q\\ q \mid b+b'}} \Psi(r; q, b)\Psi(r; q, b') .\end{split}\]
The conditions on $b$ and $b'$ in the last sum imply that $b'= q-b$ for $1\le b\le q-1$ or $b=b'=q$, and therefore we conclude, since $\Psi(r; q, q-b)=\Psi(r; q, -b)$, 
\begin{equation} \label{Sq2} \mathcal{S}(q) = \sum_{b=1}^q\Psi(r; q, b)\Psi(r; q, -b). \end{equation}

\section{Evaluating \texorpdfstring{$\mathcal{S}(q)$}{S(q)} using the Goldbach Conjecture}

Letting
\begin{equation} \label{N} r= e^{-1/N}, \end{equation}
then on taking $n=qk$ and using \eqref{psi_2odd} we have 
\[ \mathcal{S}(q) ~\overset{\eqref{Sq1}}=~ \sum_{\substack{n \\ q \mid n}} \psi_2(n)r^n = \sum_{k} \1_{2|qk} \psi_2(qk) e^{-qk/N} + O(\sum_{n}(\log^2n) e^{-n/N})=\sum_{k} \1_{2|qk} \psi_2(qk) e^{-qk/N} + O(N\log^2N).\]
From now on we specify that 
\begin{equation} \label{q<N} 1\le q\le N. \end{equation}
Letting 
\begin{equation}\label{V_q} V_q(N) := q\sum_{k} \mathfrak{S}(qk)k e^{-qk/N},\end{equation}
we conclude from \eqref{WGC} that the Weak Goldbach Conjecture implies
\begin{equation} \label{Sq1b} \delta  V_q(N) + O(N\log^2N) \leq \mathcal{S}(q) \leq (2-\delta )  V_q(N) + O(N\log^2N). \end{equation}
To evaluate $V_q(N)$, we need a formula for the singular series average
\begin{equation} G_q(x) := \sum_{k\le x} \mathfrak{S}(qk). \end{equation}
\begin{lem}[Montgomery] \label{lem1} For $x\ge 1$  we have
\begin{equation} G_q(x) = \frac{q}{\phi(q)} x + O(\frac{q}{\phi(q)}\log 2x) \end{equation}
uniformly for all $q\ge 1$.
\end{lem}
This lemma follows from \cite[Lemma 17.4]{Montgomery71}; we will give a complete proof in Section 6. 
Writing \Cref{lem1} in the form $G_q(x) = \frac{q}{\phi(q)}x + R_q(x)$ for $x\ge 1$, we obtain by partial summation
\[ \begin{split}  V_q(N) &= q\int_{1^-}^\infty ue^{-qu/N}dG_q(u)\\& =  \frac{q^2}{\phi(q)}\int_{1}^\infty ue^{-qu/N}du +q\int_{1^-}^\infty ue^{-qu/N}dR_q(u) \\&
:= I_1 +I_2.\end{split}\]
Using the condition $1\le q\le N$, we have on letting $v=qu/N$
\[ \begin{split} I_1 &= \frac{N^2}{\phi(q)}\int_{q/N}^\infty ve^{-v}dv \\&
=  \frac{N^2}{\phi(q)}\left(\int_{0}^\infty ve^{-v}dv +O(\frac{q}{N})\right) \\&
= \frac{N^2}{\phi(q)} + O(\frac{qN}{\phi(q)} ).
\end{split} \]
Next, integrating by parts and using $R_q(x)\ll \frac{q}{\phi(q)}\log(2x)$, we have
\[ \begin{split} I_2 &= O\left(\frac{q^2}{\phi(q)}e^{-q/N}\right) - q\int_1^\infty (1-\frac{q}{N}u)e^{-qu/N}R_q(u)\, du \\&
\ll \frac{q^2}{\phi(q)} + \frac{q^2}{\phi(q)}\int_1^{N/q}e^{-qu/N} \log (2u) \, du + \frac{q^3}{\phi(q)N} \int_{N/q}^\infty u  e^{-qu/N}\log(2u)\, du  \\&
\ll \frac{q^2}{\phi(q)} +\frac{q}{\phi(q)} N\log(2N/q) + \frac{qN}{\phi(q)} \int_{1}^\infty ve^{-v}\log(2Nv/q) \, dv \\ &
\ll\frac{q^2}{\phi(q)} +\frac{q}{\phi(q)} N\log(2N/q) .
\end{split} \]
Since
\begin{equation} \label{q_phi(q)}
\frac{q}{\phi(q)} \ll \log\log 3q ~\overset{\eqref{q<N}}{\ll}~ \log\log N
\end{equation}
holds, we conclude that
\begin{equation}\label{V_q=} V_q(N) = \frac{N^2}{\phi(q)} + O(N\log N\log \log N). \end{equation}
Thus by \eqref{Sq1b}
\begin{equation} \label{Sq1c} \delta \frac{N^2}{\phi(q)} + O(N\log^2 N)  \leq  \mathcal{S}(q) \leq (2-\delta )\frac{N^2}{\phi(q)} + O(N\log^2 N). \end{equation}

\section{Evaluating \texorpdfstring{$\mathcal{S}(q)$}{S(q)} using the Prime Number Theorem for Arithmetic Progressions}

Let
\begin{equation} \psi(x;q,a) := \sum_{\substack{n\le x\\ n\equiv a \,({\rm mod}\, q)}}\Lambda(n). \end{equation}
We will make use of the prime number theorem for arithmetic progressions for a modulus $q$ which has a possible exceptional real character $\chi_1$ as described in our theorem. By \cite[Cor. 11.17]{MontgomeryVaughan2007}, we have that there is a positive constant $c_1$ 
such that for $(a,q)=1$
\begin{equation} \label{PNTAP} \psi(x;q,a) = \frac{x}{\phi(q)} - \frac{\chi_1(a)x^{\beta_1}}{\phi(q) \beta_1} + O(xe^{-c_1\sqrt{\log x}}) .\end{equation} 
As explained in \cite[proof of Corollary 11.17]{MontgomeryVaughan2007}, a bound on $q$ such as $ q\le e^{2c_1\sqrt{\log x}}$ which is usually imposed on \eqref{PNTAP} is not needed because $\psi(x;q,a)$ trivially satisfies a smaller bound than the right-hand side of \eqref{PNTAP} when $q\ge e^{2c_1\sqrt{\log x}}$. 
Thus by partial summation we have using \eqref{Psi(r)} and \eqref{N} that, with $(b,q)=1$,
\[ \begin{split} \Psi(r; q, b) &= \sum_{\substack{n\\ n\equiv b \,({\rm mod}\, q)} }\Lambda(n) e^{-n/N}  = \int_{2^-}^\infty e^{-u/N} \, d\psi(u;q,b)  \\& 
=\int_0^\infty e^{-u/N}\left( \frac{1}{\phi(q)} - \frac{\chi_1(b)u^{\beta_1-1}}{\phi(q)}\right)\, du  + O\left(\frac{1}{N}\int_2^\infty ue^{-u/N}e^{-c_1\sqrt{\log u}}\, du\right)\\&
= \frac{N}{\phi(q)} \int_0^\infty e^{-t} \, dt + \frac{\chi_1(b)N^{\beta_1}}{\phi(q)} \int_0^\infty t^{\beta_1-1} e^{-t} \, dt+O(Ne^{-c_1\sqrt{\log N}}).
\end{split}\]
Recalling the gamma function
\begin{equation} \label{gamma} \Gamma(s) = \int_0^\infty t^{s-1}e^{-t}\, dt, \qquad \text{Re}(s)>0, \end{equation}
we conclude, for $(b,q)=1$, 
\begin{equation} \label{PsiAP} \Psi(r; q, b) =  \frac{N}{\phi(q)}  + \frac{\chi_1(b)\Gamma(\beta_1)N^{\beta_1}}{\phi(q)}+O(Ne^{-c_1\sqrt{\log N}}). \end{equation}
To apply this result to $\mathcal{S}(q)$, we first show that
\begin{equation}\label{Sq3} \mathcal{S}(q) ~\overset{\eqref{Sq2}}=~ \sum_{b=1}^q\Psi(r; q, b)\Psi(r; q, -b) = \sum_{\substack{1\le b\le q\\ (b,q)=1}}\Psi(r; q, b)\Psi(r; q, -b) +O(q(\log q\log N)^2), \end{equation}
which follows immediately from 
\begin{equation} \label{Psi(d>1)} \Psi(r;q,b)\ll \log q\log N,\qquad \text{ when} \quad (b,q)=d>1.\end{equation}
To prove this estimate, first note that when $(b,q)=d>1$, 
\[ \Psi(r;q,b) = \sum_{\substack{n\\n\equiv b \,({\rm mod}\, q)\\(b,q)=d>1} }\Lambda(n) r^n = \sum_{d\mid n}\Lambda(n) r^n = \sum_{m }\Lambda(dm)r^{dm}.\] 
Now $\Lambda(dm) =\log p$ if and only if $dm=p^j$, $j\ge 1$ while $\Lambda(dm)=0$ otherwise. Therefore 
\[ \Psi(r;q,b) \le \log q \sum_je^{-2^j/N}.\]
The estimate \eqref{Psi(d>1)} now follows from
\begin{equation} \label{sum_exp} \sum_je^{-2^j/N} \ll
\sum_{j< 2\log N} 1 + \sum_ {j\ge 2\log N} e^{-2^j/N} \ll \log N , \end{equation}
since for the second sum on the right-hand side we have $e^{-2^j/N}<1/e<1/2$ and also  $e^{-2^{j+k}/N} = (e^{-2^j/ N})^{2^k}$; therefore 
\[\sum_ {j\ge 2\log N} e^{-2^j/N}  < 1/2 + 1/2^2+1/2^4+1/2^8+\cdots < 1. \]

We now compute $\mathcal{S}(q)$. First,
by \eqref{PsiAP} when $(b,q)=1$ we obtain
\[ \begin{split} \Psi(r; q, b)\Psi(r; q, -b) = \frac{N^2}{\phi(q)^2} &+ \frac{N^{1+\beta_1}\Gamma(\beta_1)}{\phi(q)^2} (\chi_1(b) +\chi_1(-b)) +   \frac{\chi_1(b)\chi_1(-b)\Gamma(\beta_1)^2N^{2\beta_1}}{\phi(q)^2} \\
&+ O\left(\frac{N^2e^{-c_1\sqrt{\log N}}}{\phi(q)}\right) 
+ O\left(N^2e^{-2c_1\sqrt{\log N}}\right). \end{split} \]
Thus
\[  \sum_{\substack{1\le b\le q\\ (b,q)=1}}\Psi(r; q, b)\Psi(r; q, -b) =  \frac{N^2}{\phi(q)} +      \frac{\chi_1(-1)\Gamma(\beta_1)^2N^{2\beta_1}}{\phi(q)} + O\left(N^2e^{-c_1\sqrt{\log N}}\right) 
+ O\left(\phi(q)N^2e^{-2c_1\sqrt{\log N}}\right),  \]
where we used the fact that the sum of a Dirichlet character over a reduced residue class vanishes, and that $\chi_1$ is a real character so that $\chi_1(b)\chi_1(-b)=\chi_1(-1)$.
Anticipating our choice of $N$ in the next section, we now take $q\le e^{c_1\sqrt{\log N}}$. Next we apply \eqref{Sq3}, and since the error term $O(q(\log q\log N)^2)$ from that equation and the last error term above may both be absorbed into the first error term, we conclude
\begin{equation} \label{Sq4} \mathcal{S}(q) = \frac{N^2}{\phi(q)} +  \frac{\chi_1(-1)\Gamma(\beta_1)^2N^{2\beta_1}}{\phi(q)} + O(N^2e^{-c_1\sqrt{\log N}}). \end{equation}

\section{Proof of Theorem 1}

We will now prove the following more precise version of \Cref{thm1}.
\begin{theorem} \label{thm3} Assume the Weak Goldbach Conjecture with a given fixed $0<\delta<1$. For $q$ sufficiently large suppose $\chi_1$ is a real character modulo $q$ for which $L(s,\chi_1)$ has a real zero $\beta_1$ with $1 - c/\log q<\beta_1$ for a positive constant $c$. Let $c_1$ be the positive constant in the prime number theorem for arithmetic progressions \eqref{PNTAP}. Then for any fixed positive constant $c'<c_1$  we have
\begin{equation}  \beta_1 < 1 - \frac{\frac12(c')^2\log(\frac{1}{1-\delta})}{\log^2q}.\end{equation}
\end{theorem}
\begin{proof}[Proof of \Cref{thm3}]
We substitute \eqref{Sq4} into \eqref{Sq1c} and obtain
\[ \delta \frac{N^2}{\phi(q)} + O(N\log N)  \leq  \frac{N^2}{\phi(q)} +  \frac{\chi_1(-1)\Gamma(\beta_1)^2N^{2\beta_1}}{\phi(q)} + O(N^2e^{-c_1\sqrt{\log N}}) \leq (2-\delta )\frac{N^2}{\phi(q)} + O(N\log^2 N),\]
and therefore
\[ - (1-\delta) +O(\phi(q)e^{-c_1\sqrt{\log N}}) \le \chi_1(-1)\Gamma(\beta_1)^2N^{2(\beta_1-1)}\le (1-\delta) +O(\phi(q)e^{-c_1\sqrt{\log N}}).\]
If $\chi(-1) = -1$ then from the lower bound above we obtain
\begin{equation} \label{key} \Gamma(\beta_1)^2N^{2(\beta_1-1)} \le (1-\delta) +O(\phi(q)e^{-c_1\sqrt{\log N}}),\end{equation}
while if $\chi(-1)=1$ the upper bound above gives \eqref{key}. Thus \eqref{key} holds in both cases. 

Now choose $c'>0$ to be any fixed constant with $c'<c_1$ and a second constant $c''$ such that $c'<c''<c_1$. Defining  $N$ by 
\[ \log N := \left(\frac{1}{c''}\log q\right)^2, \]
so that $N\to \infty$ as $q\to \infty$.
Solving for $q$ we have $q=e^{c''\sqrt{\log N}}$, which verifies our earlier use of the inequality $q\le e^{c_1\sqrt{\log N}}$. Further notice that the error term in \eqref{key} is $o(1)$ as $q\to \infty$ since
\[ \phi(q)e^{-c_1\sqrt{\log N}} \le e^{-(c_1-c'')\sqrt{\log N}}=o(1).\]
Rewriting \eqref{key} we obtain
\[ \beta_1 \le 1+ \frac{ \log\left(\frac{1-\delta+o(1)}{\Gamma(\beta_1)^2}\right) }{2\log N} = 
1 - \frac{\frac12(c'')^2 \log\left(\frac{\Gamma(\beta_1)^2}{1-\delta +o(1)}\right) }{\log^2q}.\]
It is easy to verify from the definition of the gamma function $\Gamma(s)$ given in \eqref{gamma} that $\Gamma'(s)$ is analytic when $\text{Re}(s)>0 $ and since $1-c/\log q \le \beta_1 <1$ and $\Gamma(1)=1$ we see by the mean value theorem that $\Gamma(\beta_1)=1 +O(1/\log q)=1+o(1)$ as $q\to \infty$.
We conclude
\[ \beta_1 \le
1 - \frac{\frac12(c'')^2 \log\left(\frac{1+o(1)}{1-\delta +o(1)}\right) }{\log^2q} < 1 - \frac{\frac12(c')^2 \log\left(\frac{1}{1-\delta }\right) }{\log^2q}\]
on taking $q$ sufficiently large, which proves \Cref{thm3}.
\end{proof}

\section{ Proof of Lemma 1}

\begin{proof} We define
\begin{equation} H_q(k) := \prod_{\substack{ p \mid k \\ (p,2q)=1}} \!\left(\frac{p - 1}{p - 2}\right), \end{equation}
and take $ H(k):= H_1(k)$. Thus we can rewrite \eqref{SingProd} as 
\begin{equation} \mathfrak{S}(k) = \1_{2|k} 2C_2 H(k). \end{equation}
Since clearly
\begin{equation} H(qk) = H(q)H_q(k) ,\end{equation}
we have 
\[ G_q(x) = \sum_{k\le x} \mathfrak{S}(qk) = 2C_2 H(q) \sum_{k\le x}\1_{2|qk} H_q(k) .\]
Using $ \1_{2|qk}= \1_{2| q}+\1_{2\nmid q}\1_{2|k}$, we have
\[ G_q(x) = 2C_2 H(q)\left( \1_{2|q}\sum_{k\le x} H_q(k) + \1_{2\nmid q}\sum_{\substack{k\le x \\ 2|k}} H_q(k) \right).\]
Writing 
\begin{equation} \widetilde{G}_q(x) := \sum_{k\le x} H_q(k) ,\end{equation}
we conclude, on noting $H_q(2k) =H_q(k)$, that
\begin{equation} \label{G_qformula} G_q(x) = 2C_2 H(q)\left( \1_{2|q}\widetilde{G}_q(x) + \1_{2\nmid q}\widetilde{G}_q(x/2)\right).\end{equation}
We will prove below that 
\begin{equation}\label{desired} \widetilde{G}_q(x) = \frac{q}{(2,q)\phi(q)C_2H(q)}x + O(\log 2x), 
\end{equation}
which on substituting into \eqref{G_qformula} gives immediately
\[ G_q(x) = \frac{q}{\phi(q)}x +O(H(q)\log 2x). \]
\Cref{lem1} now follows from the estimate $H(q) \ll q/\phi(q)$ which can be verified by the calculation
\begin{equation} \label{phi-H}
    \frac{\phi(q)}{q} H(q) = \frac{1}{(2,q)}\prod_{\substack{p\mid q\\p>2}}\left( 1-\frac{1}{p}\right)\left(\frac{p-1}{p-2}\right) \le \prod_{p>2}\left(1+\frac{1}{p(p-2)}\right) \ll 1.
\end{equation}
\end{proof}
\begin{proof}[Proof of \eqref{desired}]
We give an expanded version of Montgomery's very nice proof \cite[Lemma 17.4]{Montgomery71}. First, 
\[ H_q(k) = \prod_{\substack{ p \mid k \\ (p,2q)=1}} \!\left(1+\frac{1}{p - 2}\right)= \sum_{d|k}f_q(d),\]
where 
\[  f_q(d) = \1_{(d,2q)=1}\, \mu(d)^2\prod_{p\mid d}\left( \frac1{p-2}\right).\]
Hence
\begin{equation}\label{last}  \begin{split} \widetilde{G}_q(x) &= \sum_{k\le x}\sum_{d|k}f_q(d) \\
 &= \sum_{d\le x} f_q(d)\sum_{\substack{k\le x \\ d|k}} 1   \\&
 = x\sum_{d\le x }\frac{f_q(d)}{d} + O\bigg(\sum_{d\le x} f_q(d)\bigg) \\&
 = x\sum_{d=1 }^\infty\frac{f_q(d)}{d} + O\bigg(x\sum_{d >  x} \frac{f_1(d)}{d}\bigg) + O\bigg(\sum_{d\le x} f_1(d)\bigg) .
\end{split} \end{equation}
 
Since $f_q(d)$ is multiplicative and $d$ square-free, recalling the equality \eqref{phi-H}, we have
 \[ \begin{split} \sum_{d=1 }^\infty\frac{f_q(d)}{d} &= \prod_p\left( 1+ \frac{f_q(p)}{p}\right) \\&
 =\prod_{\substack{ p\\ (p,2q)=1}}\left( 1+ \frac{1}{p(p-2)}\right) \\&
 =\prod_{p>2}\left( 1+ \frac{1}{p(p-2)}\right)\prod_{\substack{ p\mid q\\ p>2}}\left( 1+ \frac{1}{p(p-2)}\right)^{-1}\\&
 = \frac{1}{C_2}\prod_{\substack{ p\mid q\\ p>2}}\left(\frac{p(p-2)}{(p-1)^2}\right)\\&
 = \frac{q}{ (2,q)\phi(q)C_2 H(q)},
 \end{split} \]
which gives the main term in \eqref{desired}. For the two error terms above, we see that $f_1(d) = \mu(d)^2/\phi_2(d)$, where $\phi_2(p):=p-2$
and this is extended to $\phi_2(d)$ for squarefree $d$ with $(d,2)=1$ by multiplicativity. Thus
 \[ \begin{split}  S(x,f_1) &:= \sum_{d\le x}f_1(d) = \sum_{\substack{d\le x\\ (d,2)=1}}\frac{\mu(d)^2}{\phi_2(d)} \\&
 \ll \prod_{2<p\le x}\left( 1+ \frac{1}{p-2}\right) \\ &
 = \exp\left( \sum_{2<p\le x} \log\left( 1+ \frac{1}{p-2}\right)\right) \\&
 \ll \exp\left( \sum_{p\le x} \frac1{p}\right) \ll \exp(\log\log2 x +O(1)) \ll \log2 x ,
 \end{split}\]
by Mertens formula. Finally by partial summation
 \[ \sum_{d>  x} \frac{f_1(d)}{d} = \int_{x^+}^\infty \frac{dS(u,f_1)}{u} \ll \frac{\log 2 x}{x} + \int_x^\infty\frac{ \log 2 u}{u^2}\, du \ll \frac{\log 2 x}{x} .\]
Using these estimates in \eqref{last} completes the proof of \eqref{desired}.
\end{proof}

\section{Proof of Theorem 2}

Define
\begin{equation} \label{Psi_2(r,k)} 
\Psi_2(r,k) := \sum_n\Lambda(n)\Lambda(n-k)r^{2n-k}. \end{equation}
Then in place of \eqref{Psi^2} we have, on letting $k=n-n'$, 
\[ 
|\Psi_2(z)|^2 = \sum_{n,n'}\Lambda(n)\Lambda(n')r^{n+n'}e((n-n')\alpha) = \sum_{k=-\infty}^\infty \Psi_2(r,k) e(k\alpha). \]
Since $\Psi_2(r,k)=\Psi_2(r,-k)$, we have
\begin{equation} \label{|Psi|^2} 
|\Psi_2(z)|^2 = \Psi_2(r,0) + 2 {\rm Re} \sum_{k}\Psi_2(r,k) e(k\alpha). \end{equation}
Recalling  that $r=e^{-1/N}<1$ from \eqref{N}, we first show that, for odd $k\ge 1$,   
\begin{equation} \label{Psi2_k_odd}
\Psi_2(r,k) \ll e^{-k/N}\log^2N + (\log k)e^{-k/N} \log N, \end{equation}
which corresponds to \eqref{psi_2(x,k)odd}.
To prove this, note first that $\Lambda(n-k)=0$  for $k\ge n$. Next, for $k\ge 1$ odd,   $\Lambda(n)\Lambda(n-k) =0$ holds i) if $n$ is even and $n\neq 2^j$, or ii) if $n$ is odd and $n-k\neq 2^j$. Hence
\[\begin{split}
\Psi_2(r,k) &= (\log2) \sum_{k<2^j} \Lambda(2^j-k)r^{2^j+(2^j-k)}
+ (\log2)\sum_{j}\Lambda(2^j+k)r^{2^{j+1}+k}  \\
&\ll e^{-k/N} \left(\sum_{k<2^j}\log(2^j-k) e^{-(2^j-k)/N}
+ \sum_j \log(2^j+k)e^{-2^{j+1}/N}\right).
\end{split}\]
To estimate these last two sums we use $\sum_j e^{-2^j/N} \ll \log{N} $ from  \eqref{sum_exp} together with the estimate $\sum_j je^{-2^j/N} \ll \log^2{N} $
obtained immediately by the argument used to obtain \eqref{sum_exp}.
For the first sum there is one $j$ satisfying $k<2^j \le 2k$ and therefore we have
\[\sum_{k<2^j\le 2k}\log(2^j-k) e^{-(2^j-k)/N}\ll \log k. \]
For $2^j>2k$, we have $2^j-k> 2^j -2^{j-1} = 2^{j-1}$, and therefore
\[ \sum_{2k<2^j}\log(2^j-k) e^{-(2^j-k)/N}\ll \sum_j j e^{-2^{j-1}/N}\ll \log^2N. \]
For the second sum, since $\log(2^j +k) \ll j +\log k $, we have
\[ \sum_j \log(2^j+k)e^{-2^{j+1}/N} \ll \sum_j (j +\log k)e^{-2^{j}/N}\ll \log^2N + (\log k) (\log N). \]
Substituting we have
\[\Psi_2(r,k) \ll e^{-k/N}( \log^2N + (\log k) (\log N)) 
\]
which proves \eqref{Psi2_k_odd}.

Next
\[ \Psi_2(r,0) = \sum_n\Lambda(n)^2r^{2n} \le \sum_n \Lambda(n) (\log n) e^{-2n/N}.\]
Letting $\psi(x):= \sum_{n\le x}\Lambda(n)$ and using the Chebyshev bound $\psi(x) \ll x$ we have by partial summation
\begin{equation} \begin{split} \label{Psi_2(r,0)}
\Psi_2(r,0) 
&\ll \int_2^\infty \psi(u) \dfrac{d}{du} \left( (\log u)\, e^{-2u/N} \right) du \\&
\ll \int_2^\infty u e^{-2u/N} \left( \frac1u + \frac2N\log u \right) du \\&
\ll \log N \int_2^N e^{-2u/N} du + \frac1N \int_N^\infty (u\log u )e^{-2u/N} du \\&
\ll N\log N. \end{split} \end{equation}

For $k\ge 2$ even, we use partial summation with $\psi_2(x,k)$ and the upper bound conjecture \eqref{WPPC} to obtain
\begin{equation} \label{Psi_2conjecture}\begin{split}
\Psi_2(r,k) &= -\int_k^\infty \psi_2(u,k)\, \dfrac{d}{du} e^{-(2u-k)/N}\, du = \frac{2e^{-k/N}}{N} \int_k^\infty \psi_2(u,k)e^{-2(u-k)/N}\, du \\&
\le \frac{2(2-\delta)\mathfrak{S}(k)e^{-k/N}}{N} \int_k^\infty\left((u-k) +o(u)\right) e^{-2(u-k)/N}\, du \\&
=  \frac{2-\delta}{2}\mathfrak{S}(k)Ne^{-k/N} \int_0^\infty(v +o(v)+ o(k/N)) e^{-v}\, dv \\&
=  \frac{2-\delta}{2}\mathfrak{S}(k)Ne^{-k/N} \left(1+o(1)+o(k/N)\right). \end{split}\end{equation}

Corresponding to $\mathcal{S}(q)$ in \eqref{Sq}, we define
\begin{equation} \label{Tq}  \mathcal{T}(q) := \frac{1}{q} \sum_{a=1}^q |\Psi_2(re(a/q)|^2 \end{equation}
and have by \eqref{|Psi|^2} that
\[ \label{Tq1} \mathcal{T}(q)  = \Psi_2(r,0) + 2 \sum_{\substack{k \\ q \mid k}} \Psi_2(r,k)
=\Psi_2(r,0) + 2\sum_j\Psi_2(r,qj).
\]
Applying our estimates \eqref{Psi2_k_odd}, \eqref{Psi_2(r,0)}, and \eqref{Psi_2conjecture} for $\Psi_2(r,k)$, and recalling $1\le q\le N$ by \eqref{q<N},  we have
\[ \begin{split} \mathcal{T}(q) &\le O(N\log N) + (2-\delta)N \sum_j\mathfrak{S}(qj)e^{-qj/N}(1+o(1)+o(qj/N)) \\
&\qquad+ O\left((\log^2N)\sum_je^{-qj/N} + (\log{N})\sum_j \log(qj)\,e^{-qj/N}\right) \\ &
\le (2-\delta)N(1+o(1))\sum_j\mathfrak{S}(qj)e^{-qj/N} + o(V_q(N)) + O(N\log^2N), \end{split} \]
where $V_q(N)$ is defined in \eqref{V_q} and by \eqref{V_q=} $ V_q(N) = \frac{N^2}{\phi(q)} + O(N\log N\log\log N)$. We show below that 
\begin{equation} \label{Slast} \sum_j\mathfrak{S}(qj)e^{-qj/N} = \frac{N}{\phi(q)} + O(\log N\log\log N),\end{equation}
and thus we conclude
\begin{equation} \label{Tqbound}\mathcal{T}(q) \le (2-\delta)\frac{N^2}{\phi(q)}(1+o(1)) +O(N\log^2N).\end{equation}
The calculation for \eqref{Slast} is nearly the same as the earlier one for obtaining \eqref{V_q=}. We have
\[ \begin{split} \sum_j\mathfrak{S}(qj)e^{-qj/N} &= \int_{1^-}^\infty e^{-qu/N}dG_q(u)\\& =  \frac{q}{\phi(q)}\int_{1}^\infty e^{-qu/N}du +\int_{1^-}^\infty e^{-qu/N}dR_q(u) \\& =\frac{N}{\phi(q)}\int_{q/N}^\infty e^{-v}dv +O\left(\frac{q^2}{\phi(q)N}\int_{1}^\infty e^{-qu/N}\log 2u\, du\right) \\ &
= \frac{N}{\phi(q)} +O(\frac{q}{\log q}) + O(\frac{q}{\phi(q)}\log(2N/q)) = \frac{N}{\phi(q)} + O(\frac{q}{\phi(q)}\log N).
\end{split} \]
Combining this with \eqref{q_phi(q)} gives \eqref{Slast}.

The evaluation of $\mathcal{T}(q)$ using the prime number theorem for arithmetic progressions is almost identical to the earlier evaluation of $\mathcal{S}(q)$. 
Using 
\[ \Psi(re(a/q)) = \sum_{b=1}^q e(ab/q)\Psi(r; q, b) \]
in \eqref{Tq}, we have
\[ \begin{split}  \mathcal{T}(q) &= \frac{1}{q} \sum_{a=1}^q \sum_{1\le b,b'\le q} e(a(b-b')/q)\Psi(r; q, b)\Psi(r; q, b') \\ &
= \sum_{\substack{1\le b,b'\le q\\ q \mid b-b'}} \Psi(r; q, b)\Psi(r; q, b') \\&
=\sum_{b=1}^q\Psi(r; q, b)^2 = \sum_{\substack{1\le b\le q\\ (b,q)=1}}\Psi(r; q, b)^2 +O(q(\log q\log N)^2),\end{split}\]
on using \eqref{Psi(d>1)}. From \eqref{PsiAP} we obtain
\[ \Psi(r; q, b)^2 = \frac{N^2}{\phi(q)^2} + \frac{2\chi_1(b)N^{1+\beta_1}\Gamma(\beta_1)}{\phi(q)^2} +   \frac{\Gamma(\beta_1)^2N^{2\beta_1}}{\phi(q)^2} 
+ O\left(\frac{N^2e^{-c_1\sqrt{\log N}}}{\phi(q)}\right) 
+ O\left(N^2e^{-2c_1\sqrt{\log N}}\right),\]
and substituting this into the previous equation and noting the term with $\chi_1(b)$ vanishes when summed, we obtain as before
\begin{equation} \label{Tq4} \mathcal{T}(q) = \frac{N^2}{\phi(q)} +  \frac{\Gamma(\beta_1)^2N^{2\beta_1}}{\phi(q)} + O(N^2e^{-c_1\sqrt{\log N}}), \end{equation} 
which is the same as the result for $\mathcal{S}(q)$ when $\chi(-1)=1$.
Combining \eqref{Tqbound} and \eqref{Tq4} we obtain \eqref{key}. The proof of \Cref{thm3} (and thus of \Cref{thm1}) then follows as before. This concludes the proof of \Cref{thm2}.


\end{document}